\documentclass[11pt,a4paper]{amsart}

\usepackage{graphicx}      
\usepackage{amsmath}
\usepackage{amssymb}
\usepackage{xcolor}
\usepackage{amsthm}
\usepackage{appendix}

\newtheorem{thm}{Theorem}

\newtheorem{lem}[thm]{Lemma}
\newtheorem{prop}[thm]{Proposition}
\newtheorem{assum}[thm]{Assumption}

\theoremstyle{definition}
\newtheorem{defn}[thm]{Definition}
\newtheorem{rem}[thm]{Remark}

\newcommand{\R}{\ensuremath{\mathbb R}}    
\newcommand{\N}{\ensuremath{\mathbb N}}    
\newcommand{\conv}{\operatorname{conv}}
\newcommand{\dist}{\operatorname{dist}}
\newcommand{\calT}{\mathcal T}

\begin{document}
\title{Optimal control of thermodynamic port-Hamiltonian Systems}
\author[B.\ Maschke, F.\ Philipp, M.\ Schaller, K.\ Worthmann, T.\ Faulwasser]{Bernhard Maschke$^{1}$, Friedrich Philipp$^{2}$, Manuel Schaller$^{2}$, Karl Worthmann$^{2}$ and Timm Faulwasser$^{3}$}
	\thanks{}
	\thanks{$^{1}$Univ Lyon, Universit{\'e} Claude Bernard Lyon 1, CNRS, LAGEPP UMR 
		5007, France (e-mail: bernhard.maschke@univ-lyon1.fr).}
	\thanks{$^{2}$Technische Universit\"at Ilmemau, Institute for Mathematics, Germany (e-mail: \{friedrich.philipp,manuel.schaller,karl.worthmann\}@tu-ilmenau.de)..}
		\thanks{$^{3}$TU Dortmund University, Institute of Energy Systems, Energy Efficiency and Energy Economics, Germany
			(e-mail: timm.faulwasser@ieee.org).}

	\thanks{{\bf Acknowledgments: }F.\ Philipp was funded by the Carl Zeiss Foundation within the project \textit{DeepTurb---Deep Learning in und von Turbulenz}.	K.\ Worthmann gratefully acknowledges funding by the German Research Foundation (DFG; grant WO\ 2056/6-1, project number 406141926). This research started during a research stay of FP and MS at the group of BM in Lyon. FP and MS thank the University of Lyon and the work group LAGEPP for the warm hospitality. MS further gratefully acknowledges funding by the French Embassy in Germany by means of a PROCOPE mobility grant. }.

\begin{abstract}
We consider the problem of minimizing the entropy, energy, or exergy production for state transitions of irreversible
port-Hamiltonian systems subject to control constraints. Via a
dissipativity-based analysis we show that optimal solutions exhibit the
manifold turnpike phenomenon with respect to the manifold of thermodynamic
equilibria. We illustrate our analytical findings via numerical results for 
a heat exchanger.

\end{abstract}

\maketitle

\section{Introduction}

The Hamiltonian formulation of controlled thermodynamic systems is a very active research area with various considered settings ranging from port-Hamiltonian (pH) systems defined on contact manifolds \cite{Eberard_RMP07,Favache_IEEE_TAC_09,Favache_ChemEngSci10,Ramirez_SCL13,Ramirez_IEEE_TAC_17}, symplectic manifolds \cite{Entropy_2018,Maschke_IFAC_NOLCOS_19_HomHamContr} to dissipative Hamiltonian (or gradient-Hamiltonian) systems such as GENERIC \cite{Oettinger_PhysRevE_06,Hoang_JPC_11,Hoang_JPC_2012}.

With respect to control of thermodynamic pH systems, there are several
papers considering passivity-based feedback stabilization  using thermodynamic
potentials, e.g., the internal energy, eventually augmented with the mechanical or electromagnetic energy
or the exergy \cite{Sangi_Energy_19_Review2ndLawControl,Sieniutycz_PhysRep00_HJBContrExergy}, the availability function \cite{Hoang_JPC_2012,Hoang_JPC_11,GarciaSandoval_JPC_17}, or the entropy creation function associated with the irreversible phenomena \cite{Ramirez_Automatica16}- 
The works mentioned above consider classical approaches for pH systems such as passivity-based damping assignment (IDA-PBC), energy function or entropy creation shaping and they do not investigate optimal control. 

However, outside the realm of pH systems, there is a series of papers considering dynamic control based on thermodynamic potentials such as the availability function \cite{Alonso96,Ydstie02,Ruszkowski05,Wang_LHMNLC15}, and also the entropy creation \cite{GarciaSandoval_ChemEngSci_16,GarciaSandoval_JPC_17}. Moreover, there is some effort regarding  entropy optimization at steady state for distributed parameter systems  \cite{Johannessen_Energy_04_OptContrEntropy}.

In this work, we consider optimal control of irreversible pH systems \cite{Ramirez_EJC13,Ramirez_ChemEngSci13} accounting both for the energy conservation and the entropy creation.
Thereby, we extend our previous works on minimzing the energy supply in the linear pH setting \cite{Schaller2020a,Philipp2021,Faulwasser2021} to the class of nonlinear irreversible pH systems, that is, we aim for a state transition with minimal energy, entropy, or exergy supply, respectively linear combinations thereof. We prove that, under suitable assumptions, a state transition that is optimal with respect to these metrics is always performed in a neighborhood of the set of thermodynamical equilibria for the majority of the time. We embed this intuitive property into the framework of manifold turnpike properties of optimal control problems.\\

\noindent\textbf{Notation.} We denote the gradient of a scalar-valued function $F:\R^n \to \R$ by $\frac{\partial F}{\partial x}$ or just $F_x$ and define the {\em Poisson bracket} $\{S,H\}_J$ of two functions $S,H : \mathbb R^n\to\mathbb R$ with respect to a matrix $J\in\mathbb{R}^{n\times n}$ by $\{S,H\}_{J}(x) = S_x(x)^\top J H_x(x)$.

\section{Irreversible port-Hamiltonian systems}

\noindent We commence our analysis by introducing the considered class of systems. The state space is given by $\R^n$, $n\in \N$, and, as usual in pH Systems, the input and output spaces coincide and are given by $\R^m$, $m \in \N$.

The definition of an irreversible pH system (IPHS) was introduced
by~\cite{Ramirez_ChemEngSci13}, and we slightly adapt it to our setting.
\begin{defn}\label{def:RIPHS}
	An \emph{irreversible port-Hamiltonian system} 
	is defined by the \textit{state equation}
	\begin{align}\label{eq:RIPHS}\tag{IPHS}
	\begin{split}
	\dot{x}(t)=\gamma(x(t),H_x(x(t)))\{S,H\}_{J}(x(t))J H_x(x(t))+g(x,H_x(x))u(t).
	\end{split}
	\end{align}
	with a skew-symmetric structure matrix $J \in \mathbb{R}^{n\times n}$, a strictly positive continuous function $\gamma : \mathbb R^{2n}\to\mathbb R$, 
	\begin{itemize}
		\item[(i)] 
		a continuously differentiable \emph{non-negative Hamiltonian 
			function} $H: \R^n \to \R_0^+$, and  
		\item[(ii)] 
		an \emph{entropy 
			function} $S: \R^n \to \R $. 
	\end{itemize}
\end{defn}
The system is completed with two output variables, the
energy-conjugated output $y_H$ and the entropy-conjugated output $y_S$ defined by
\begin{align}\label{eq:outputs}
y_H := g(x,H_x)^\top H_x \quad\text{ and }\quad y_S := g(x,H_x)^\top S_x.
\end{align}
Direct calculations show that every trajectory~$x$ satisfying the dynamics~\eqref{eq:RIPHS} obeys the energy and entropy balance
\begin{align}\label{eq:balance}
\begin{split}
\tfrac{\text{d}}{\text{d}t}H(x(t)) &= y_H(t)^\top u(t),\\
\tfrac{\text{d}}{\text{d}t}S(x(t)) &= \gamma(x(t),H_x(x(t))) \{S,H\}_J^2(x(t)) + y_S(t)^\top u(t).
\end{split}
\end{align}
Here, $y_H(t)^\top u(t)$ represents the energy flow, i.e., the power supplied to/extracted from the system, whereas $y_S(t)^\top u(t)$ is to be interpreted as the entropy flow injected into/drawn from the system. For closed systems, i.e., $u\equiv 0$, it can be directly inferred from the equations in~\eqref{eq:balance} and the positivity of~$\gamma$ that energy is preserved and entropy is non-decreasing. Hence, the two fundamental laws of thermodynamics hold. In particular, the entropy balance captures the irreversible nature of~\eqref{eq:RIPHS}.

\subsection{Example: Heat exchanger}\label{subsec:heat}

In this subsection, we introduce the heat exchanger as depicted in Figure~\ref{fig:heat} in order to motivate the problem formulation and to illustrate our results. The example is slightly adapted from \cite{RamirezThesis:2012}.
\begin{figure}[htb]
	\includegraphics[width=0.6\columnwidth]{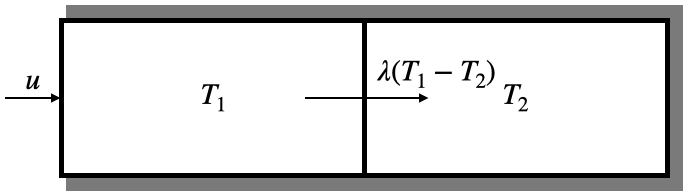}
	\caption{Model of a heat exchanger with two compartments}\label{fig:heat}
\end{figure}


The thermodynamic properties of each compartment, under the assumption that the walls are non deformable and impermeable, are given by the relation between temperature and entropy
$$
T_i(S_i) = T_\text{ref}\cdot e^{(S_i-S_\text{ref})/c_i}, \qquad i=1,2,
$$
where $S_\text{ref}\in \mathbb{R}$ is a reference entropy corresponding to the reference temperature $T_\text{ref}$ and $c_i$, $i=1,2$, are heat capacities, cf.\ \cite[Section 2.2]{Couenne06}. The energy in each compartment, denoted by $H_i(S_i)$, $i=1,2$, can be obtained by integrating Gibbs' equation $\text{d}H_i = T_i\text{d} S_i$ as a primitive of the function $T_i(S_i)$, $i=1,2$.

The state vector of the IPHS is composed of the entropies of the compartments $x=(S_1,S_2)^\top$ and the total energy (entropy) is given by the sum of the energies (entropies, resp.) in the compartments, i.e.,
\begin{align*}
H(x) = H_1(S_1) + H_2(S_2),\quad  S(x) = S_1+S_2 = (1,1)^\top x.
\end{align*}
Consider first the situation where the two compartments are isolated and the dynamics only arise from the heat flow through the wall separating the two compartments. Due to Fourier’s law, the heat flow is given by
\begin{align}
\label{e:fourier}
\dot{Q} = \lambda(T_1-T_2),
\end{align}
where $\lambda> 0$ is a heat conduction coefficient.
By continuity of the heat flux, we have
\begin{align*}
\dot{Q} = -\tfrac{\text{d}}{\text{d}t} H_1(S_1(t)) = \tfrac{\text{d}}{\text{d}t}H_2(S_2(t))
\intertext{and hence}
\lambda(T_1-T_2)=-T_1\tfrac{\text{d}}{\text{d}t}S_1(t) = T_2\tfrac{\text{d}}{\text{d}t}S_2(t),
\end{align*}
which yields the entropy balance equations for each compartment, written as follows
\begin{align*}
\tfrac{\text{d}}{\text{d}t} \begin{pmatrix}
S_1(t)\\
S_2(t)
\end{pmatrix}
= \lambda\left(\tfrac{1}{T_2(t)}-\tfrac{1}{T_1(t)}\right)\underbrace{\begin{pmatrix}
	0&-1\\1&0
	\end{pmatrix}}_{=:J}
\begin{pmatrix}
T_1(t)\\
T_2(t)
\end{pmatrix},
\end{align*}
hereby matching the definition of the drift term of $\eqref{eq:RIPHS}$ with
\begin{align*}
\gamma(x,\tfrac{\partial H}{\partial x}) = \tfrac{\lambda}{T_1T_2}, \qquad \{S,H\}_J =\begin{pmatrix}
1&1
\end{pmatrix} J \begin{pmatrix}
T_1\\
T_2
\end{pmatrix} = T_1-T_2.
\end{align*}
\textbf{Entropy flow control.} The canonical choice of an input would be to consider the entropy flowing into or out of compartment one. In this case, we have
\begin{align}\label{eq:entcont}
\frac{\text{d}}{\text{d}t} \begin{pmatrix}
S_1\\
S_2
\end{pmatrix}  &= \gamma(x,H_x) \{S,H\}_J(x) J H_x + \begin{pmatrix}
1\\0
\end{pmatrix}u.
\end{align}
\textbf{Control by a thermostat.} The realizable choice is to induce a heat flow through the external wall connecting compartment two to a thermostat at a controlled temperature $T_e$. If, e.g., compartment one is not isolated, the heat flow between this compartment and the environment can be described via
\begin{align*}
\dot{Q}_e = \lambda_e(T_e-T_2),
\end{align*}
with $\lambda_e>0$ being a heat conduction coefficient.
Thus, the energy balance in the first compartment reads
\begin{align*}
T_1\tfrac{\text{d}}{\text{d}t}S_1(t) = -\lambda(T_1(t)-T_2(t)) +  \lambda_e(T_e(t)-T_1(t))
\end{align*}
and we obtain the dynamics
\begin{align}\label{eq:tempcont}
\frac{\text{d}}{\text{d}t} \begin{pmatrix}
S_1\\
S_2
\end{pmatrix}  &= \gamma(x,H_x) \{S,H\}_J(x) J H_x + \lambda_e \begin{pmatrix}
\tfrac{T_e-T_1}{T_1}\\0
\end{pmatrix}.
\end{align}
Note that the input map does not correspond to Definition~\ref{def:RIPHS}, as it is affine in the control: $W(x,H_x)+g(x,H_x)u$ as in \cite{Ramirez_ChemEngSci13}. In the sequel, we shall consider the control problem using as input the entropy flow into compartment one \eqref{eq:entcont} which is related to the thermostat temperature control in \eqref{eq:tempcont} by the state dependent control transformation $u \to \tfrac{u - T_1}{T_1}$. 

\subsection{Optimal Control Problem for IPHS}

\noindent Considering an optimization horizon~$t_f \geq 0$, an initial value~$x^0\in \mathbb{R}^n$, and a terminal region $\Psi \subset \R^n$, we obtain the prototypical Optimal Control Problem (OCP)
\begin{align}\label{eq:phOCP}\tag{phOCP}
\begin{split}
\min_{u\in L^\infty(0,t_f;\mathbb{U})} \int_0^{t_f}\!\!\!&\left[\alpha_1y_H(t) - \alpha_2 T_0y_S(t)\right]^\top u(t)\,\text{d}t \\
\text{s.t. } \eqref{eq:RIPHS},\quad &x(0)=x^0, \quad x(t_f)\in \Psi.
\end{split}
\end{align}
Here, $y_H$ and~$y_S$ are given by~\eqref{eq:outputs}. The set of admissible control values $\mathbb{U} \subset \mathbb{R}^m$ is supposed to contain the origin and to be compact and convex. 
Then, the feasible set, i.e., the set of all admissible control functions, is given by 
\begin{equation*}
\mathcal{U}_{t_f} := \{u \in L^\infty(0,t_f,\mathbb{R}^m) \,|\, u(t) \in \mathbb{U}, x(t_f; x^0,u) \in \Psi\}.
\end{equation*}
In the cost functional, $T_0>0$ is a fixed scalar reference temperature and the coefficients $\alpha_1,\alpha_2 \in [0,1]$ with $\alpha_1 + \alpha_2 = 1$ yield a convex combination of the energy flow and the entropy flow. The three most important cases are the following:
\begin{itemize}
	\item minimal energy supply, i.e., $\alpha_1 = 1$, $\alpha_2=0$,
	\item minimal entropy extraction, i.e., $\alpha_1 = 0$, $\alpha_2=1$,
	\item minimal exergy supply, i.e., $\alpha_1 = \alpha_2 = \frac 12$.
\end{itemize}

Setting $\ell_{\alpha_1,\alpha_2}(x,u) = [\alpha_1y_H - \alpha_2 T_0 y_s]^\top u$ and using the balance equations \eqref{eq:balance}, we obtain the identity
\begin{align}
\begin{split}
\label{eq:energybalance}
\int_0^{t_f} \ell_{\alpha_1,\alpha_2}(x,u)\,\text{d}t &= \alpha_1\left[H(x(t_f))-H(x(0))\right]\\ &+ \alpha_2T_0\left(S(x(0))- S(x(t_f)) + \int_0^{t_f}\gamma(x,\tfrac{\partial H}{\partial x})\{S,H\}^2_J\,\text{d}t\right).
\end{split}
\end{align}

The following proposition directly follows from \cite[Theorem IV.2]{Macki2012}. 

\begin{prop}[Existence of solutions]
	Let $\Psi = \{0\}$, $x^0\in \R^n$ and assume that $\mathcal{U}_{t_f} \neq \emptyset$, i.e., there is a feasible control. If state responses to admissible controls are bounded, i.e., for any $t_f$ there is $\alpha \geq 0$ such that
	\begin{align*}
	\|x(t;x^0,u)\| \leq \alpha \quad \forall u\in \mathcal{U}_T,\,\,\, 0\leq t\leq t_f,
	\end{align*}
	then there exists an optimal control.
\end{prop}

\noindent The assumption of a bounded state response is satisfied if, e.g., one has exponential stability of the uncontrolled dynamics and uniform boundedness of the input map $g$, see \cite[Theorem 2.3]{Saka2021}. 
%

\section{Turnpikes towards the manifold of thermodynamic equilibria}
\noindent To prove the main result of this paper, we will impose the following assumptions on the energy and entropy function.

\begin{assum}\label{as:ham}
	Let the following hold.
	\begin{enumerate}
		\item[(a)] $H\in C^2(\R^n,\R)$,
		\item[(b)] $H_x: \R^n\to\R^n$ is a diffeomorphism,
		\item[(c)] the entropy function is linear in the state, i.e., ${S(x)=l^\top x}$ with some $l\in \R^n$.
	\end{enumerate}
\end{assum}

\noindent We briefly discuss the above assumptions. In thermodynamic pH-systems energy $H(x)$ is mostly non-quadratic and, whereas (b) is satisfied for many thermodynamic pH systems, the norm on the inverse of $H_x$ is usually not uniformly bounded on $\R^n$, cf.\ the heat exchanger in Section~\ref{subsec:heat}. The linearity of the entropy function in the state is satisfied by convention for all common models of IPHS, as the total entropy $S$ or the entropies in the subdomains $S_i$, $i=1,2$ as in the heat exchanger are also considered as a state itself. For various examples we refer to \cite{RamirezThesis:2012,Ramirez_ChemEngSci13}.

We define the set of {\em thermodynamic equilibria} by
\begin{align*}
\mathcal{T} := \big\{x\in \mathbb{R}^n : \gamma(x,H_x(x))\{S,H\}^2_{J}(x)=0\big\} = \big\{x\in\R^n : \{S,H\}_{J}(x)=0\},
\end{align*}
where $\{S,H\}_{J}(x)$ has the physical interpretation of the driving force of the irreversible phenomenon \cite{Ramirez_ChemEngSci13}.
In view of linearity of the entropy function (Assumption~\ref{as:ham}(c)), we compute
\begin{align*}
\{S,H\}_{J}(x) = S_x(x)^\top J H_x(x) =  l^\top J H_x(x) = -H_x(x)^\top Jl
\end{align*}
and hence $ \{S,H\}_{J}(x) = 0$ whenever $H_x(x) \in (Jl)^\perp$, i.e., 
$$
\mathcal{T} = \tfrac{\partial H}{\partial x}^{-1}\left((Jl)^\perp\right).
$$
Further, by differentiating the above expression and due to pointwise invertibility of $H_{xx}$ (Assumption~\ref{as:ham}(b)), the set $\mathcal{T}$ is a manifold. This manifold is $n$-dimensional if $Jl=0$ and $n-1$-dimensional otherwise.

We now recall the manifold turnpike property as introduced by \cite{Faulwasser2021b} in the context of trim manifolds for Lagrangian and Hamiltonian mechanical systems. For thermodynamic IPHS, the manifold of interest is given by the thermodynamic equilibria.
\begin{defn}[Integral state manifold turnpike property]
	Let $\ell\in C^1(\mathbb{R}^{n+m})$, $\varphi\in C^1(\R^n)$, $\Psi\subset\R^n$ be closed and $f\in C^1(\R^{n+m})$. We say that a general OCP of the form
	\begin{align}
	\begin{split}\label{e:lin_OCP}
	\min_{u\in L^1(0,T;\mathbb{U})}\,&\varphi(x(T)) + \int_0^{t_f} \ell(x(t),u(t))\,dt\\
	\text{s.t. }\dot x = &f(x,u),\quad x(0)=x^0,\quad x(t_f)\in\Psi
	\end{split}
	\end{align}
	has the {\em integral state turnpike property} on a set $S_{\rm tp}\subset\mathbb{R}^n$ with respect to a manifold $\mathcal{T}\subset\R^n$, if for all compact $K\subset S_{\rm tp}$ there is a constant $C_K$ such that for all $x^0\in K$ and $t_f>0$, each optimal pair $(x^\star ,u^\star )$ satisfies 
	\begin{align}\label{e:integral_tp}
	\int_0^{t_f}\dist^2\big(x^\star (t),\mathcal{T}\big)\,dt\le C_K.
	\end{align}
	
\end{defn}
The previous definition can be interpreted as follows. As the upper bound in \eqref{e:integral_tp} is bounded uniformly in $t_f$, for large time horizons $t_f$, the positive integrand $\dist^2(x^*(t),\mathcal{T})$ has to be small for the majority of the time. More precisely, for $x^0 \in K\subset S_\text{tp}$ and $\varepsilon > 0$ we have
\begin{align*}
\mu\big(\{t\in [0,T] : \dist(x^\star (t),\calT) > \varepsilon\}\big)\le\tfrac 1{\varepsilon^2}\!\int_0^T\dist^2(x^\star (t),\calT)\,\text{d}t\le\tfrac{C_K}{\varepsilon^2}, 
\end{align*}
where $\mu$ denotes the standard Lebesgue measure. This behavior of optimal trajectories is called {\em measure turnpike}, cf.\ e.g.\ \cite[Definition 2]{FaulGrun22}.

\begin{lem}\label{lem:manifold_distance}
	For any compact subset $K\subset\mathbb{R}^n$, there are positive constants $c,C>0$ such that
	\begin{align*}
	c\operatorname{dist}(x,\mathcal{T})^2 \,\le\, \gamma(x,H_x(x))\{S,H\}^2_{J}(x) \,\le\, C\operatorname{dist}(x,\mathcal{T})^2
	\end{align*}
	for all $x\in K$.
\end{lem}
\begin{proof}
	The proof is given in the appendix.
\end{proof}
%
As it is typical for thermodynamic systems, neither the Hessian $H_{xx}$ of the Hamiltonian nor its inverse $H_{xx}(\cdot)^{-1}$ are uniformly bounded. For this reason the norm-like equivalence of Lemma~\ref{lem:manifold_distance} could only be proved on compact sets. 
To apply the above result to optimal trajectories and to render the involved constants uniform in the horizon, we henceforth assume that optimal trajectories are uniformly bounded in the horizon. As we will see later in our heat exchanger example, this property is verified. We comment  on future research w.r.t.\ this assumption in Section~\ref{subsec:disc}.

\begin{assum}\label{as:comp}
	%
	For any compact set of initial values $X^0$, there is a compact set $K\subset \R^n$ such that for all horizons $t_f$, the corresponding optimal state of \eqref{eq:phOCP} with initial datum $x^0\in X^0$ and horizon $t_f$ is contained in $K$, i.e., 
	\begin{align*}
	&\forall\,(x^0,t_f) \in X^0 \times (0,\infty):\\
	&\forall\,u^\star \text{optimal for }  \eqref{eq:phOCP}: x(t;x^0,u^\star) \in K\,\,\forall\, t \in [0,t_f]
	\end{align*}
	
\end{assum}

To prove the turnpike property we utilize the following notation for initial states that can first be steered to the manifold $\mathcal{T}$ and then to the terminal set $\Psi$.
\begin{align*}
\mathcal{C}(\mathcal{T},\Psi) :=\{x^0\in \mathbb{R}^n \,|\,&\exists t_1\geq 0,u_1 \in L^{\infty}(0,t_{1};\mathbb{U}) \text{ s.t. } x(t_1,u_1,x^0) \in \mathcal{T}, \\
&\exists t_2\geq 0,u_2 \in L^{\infty}(0,t_{2};\mathbb{U}) \text{ s.t. } x(t_2,u_2,x(t_1,u_1,x^0))\in \Psi\}
\end{align*}

\noindent In the following, we denote by $x(\cdot,x^0,u)$ the trajectory emanating from an initial value $x^0\in \R^n$ when applying a control $u\in L^\infty(0,T;\R^m)$. We first provide a result w.r.t.\ steady states in preparation of our turnpike theorem.

\begin{lem}
	\label{lem:ss}
	Every thermodynamic equilibrium is a controlled steady state when choosing $u\equiv 0$, that is, for all $\bar{x}\in \mathcal{T}$
	\begin{align*}
	x(t;\bar{x},0) = \bar{x} \qquad \forall t\geq 0.
	\end{align*}
\end{lem}
\begin{proof}
	The proof immediately follows from the fact that the right-hand side of the dynamics \eqref{eq:RIPHS} vanishes when plugging in $(\bar{x},0)$.	
\end{proof}
We now state and prove the main result of this paper.

\begin{thm}
	\label{thm:turnpike1}
	Let Assumption~\ref{as:comp} hold.
	Then \eqref{eq:phOCP} has the integral manifold turnpike on 
	$S_\text{tp} = \mathcal{C}(\mathcal{T},\Psi)$ with respect to the manifold of thermodynamic equilibria $\mathcal{T}$.
\end{thm}
\begin{proof}
	Let $(x^*,u^*)$ be an optimal state-control pair of \eqref{eq:phOCP}. Then, by optimality, any control $u\in L^\infty(0,t_f,\mathbb{U})$ with corresponding state trajectory $x= x(\cdot,u,x^0)$ satisfies
	\begin{align*}
	\int_0^{t_f} \ell_{\alpha_1,\alpha_2}(x^*(t),u^*(t))\,\text{d}t \leq  \int_0^{t_f} \ell_{\alpha_1,\alpha_2}(x(t),u(t))\,\text{d}t.
	\end{align*}
	Abbreviating $R(x) := \gamma(x,H_x(x)) \{S,H\}^2_J(x)$, invoking \eqref{eq:energybalance} on both sides of the above inequality and cancelling the terms depending on $x^0$ we obtain
	\begin{align*}
	&\alpha_1H(x^*(t_f)) + \alpha_2 T_0\left(-S(x^*(t_f)) + \int_0^{t_f} R(x^*(t))\,\text{d}t\right)\\
&\qquad \qquad\qquad\qquad \leq \alpha_1H(x(t_f)) + \alpha_2 T_0\left(-S(x(t_f)) + \int_0^{t_f} R(x(t))\,\text{d}t\right).
	\end{align*}
	We now will construct a suitable control such that the right-hand side is bounded uniformly in $t_f$. To this end, by $x^0 \in \mathcal{C}(\mathcal{T},\Psi)$, we get the existence of $t_1,t_2>0$ and corresponding controls $u_1,u_2$ that steer the initial state into the manifold and to the terminal region, respectively. W.l.o.g., we can assume that $t_f \geq t_1+t_2$, as the turnpike inequality \eqref{e:integral_tp} is of purely asymptotic nature in the horizon and choose \begin{align*}
	u(t):=\begin{cases}
	u_1(t) \qquad &t\in [0,t_1]\\
	0 \qquad& t \in (t_1,t_f-t_2)\\
	u_2(t-(t_f-t_2)) \qquad &t\in [t_f-t_2,t_f].
	\end{cases}
	\end{align*}
	Thus, using Lemma~\ref{lem:ss}, we have $x(t,x^0,u) = \bar{x}\in \mathcal{T}$ for all $t\in (t_1,t_f-t_2)$. Hence, abbreviating $x(t) = x(t,x^0,u)$, we have
	\begin{align*}
	\int_0^{t_f}R(x(t))\,\text{d}t 
	= \int_0^{t_1}R(x(t))\,\text{d}t +\int_{t_f-t_2}^{t_2}R(x(t))\,\text{d}t 
	\leq c_1(t_1,t_2,u_1,u_2,x^0,\bar{x}).
	\end{align*}
	Using non-negativity of the Hamiltonian from below we get
	\begin{align*}
	\int_0^{t_f}R(x^*(t))\,\text{d}t \leq \tfrac{1}{\alpha_2T_0}\left(\alpha_1H(x(t_f))+c_1\right)+ \|l\|\left(\|x^*(t_f)\| + \|x(t_f)\|\right) 
	\end{align*}
	Further, by uniform boundedness of the trajectories in the horizon $t_f$, cf.\ Assumption~\ref{as:comp}, we conclude
	\begin{align*}
	\int_0^{t_f}\gamma(x^*(t),H_x(x^*(t))) \{S,H\}^2_J(x^*(t))\,\text{d}t=  \int_0^{t_f}R(x^*(t))\,\text{d}t \leq c_2
	\end{align*}
	with a constant $c_2>0$ independent of $t_f$. Denoting by $c$ the constant appearing in the lower bound of the estimate in Lemma~\ref{lem:manifold_distance}, we obtain
	\begin{align*}
	\int_0^{t_f}\dist\left(x^*(t),\mathcal{T}\right)\,\text{d}t \leq \tfrac{c_2}{c},
	\end{align*}
	which is the turnpike property.
\end{proof}

\begin{rem}[Relation to optimal steady states]
	\label{rem:ss}
	\noindent The steady-state problem corresponding to \eqref{eq:phOCP} reads
	\begin{align*}
	&\min_{(x,u)\in  \mathbb{R}^n\times \mathbb{U}} (\alpha_1y_H + \alpha_2T_0y_S)^\top u\\
	\text{s.t. }0 &= \gamma\left(x,\tfrac{\partial H}{\partial x}\right)\left\{ S,H\right\} _{J}J\frac{\partial H}{\partial x}(x)+g\left(x,\tfrac{\partial H}{\partial x}\right)u.
	\end{align*}
	It can be immediately seen that the choice $u=0$ leads to zero cost. Further, as the energy balance \eqref{eq:energybalance} also holds for steady states this problem is equivalent to 
	\begin{align*}
	&\min_{(x,u)\in  \mathbb{R}^n\times \mathbb{U}} \alpha_2T_0\gamma\left(x,\tfrac{\partial H}{\partial x}\right) \left\{ S,H\right\}_J^2\\\
	\text{s.t. }0 &= \gamma\left(x,\tfrac{\partial H}{\partial x}\right)\left\{ S,H\right\} _{J}J\frac{\partial H}{\partial x}(x)+g\left(x,\tfrac{\partial H}{\partial x}\right)u.
	\end{align*}
	Any optimal control $u$ has to have a cost of at most zero, which means by nonnegativity of the cost functional that $\alpha_2T_0\gamma\left(x,\tfrac{\partial H}{\partial x}\right) \left\{ S,H\right\}_J^2=0$, i.e, the optimal state is contained in $\mathcal{T}$. Hence, in order to be feasible, the corresponding optimal control has to satisfy $g(x,\tfrac{\partial H}{\partial x})u=0$.
	Thus, the set of optimal steady states is given by
	\begin{align*}
	&\{(x,u)\in \mathcal{T}\times \mathbb{U}\,|\,g(x,H_x(x))u=0\}.
	\end{align*}

\end{rem}

\section{Set-point transition for a heat exchanger}
\label{sec:numres}
In this part, we present a numerical case study for the heat exchanger from Subsection~\ref{subsec:heat} with entropy flow control~\eqref{eq:entcont} to illustrate the manifold turnpike result of Theorem~\ref{thm:turnpike1}. We note that, here, the manifold of thermodynamic equilibria is actually a subspace as $\{S,H\}_J(x) = T_1(S_1)-T_2(S_2)$ and thus
$$
\mathcal{T} = \{(S_1,S_2) : T_1(S_1) = T_2(S_2)\} = \{(S_1,S_2) : S_1=S_2\}
$$
where the last equality follows from the injectivity of the exponential function that defines the temperature-entropy relation.

In this part, we set $T_\text{ref}=c_1=c_2=1$ and $S_\text{ref} = 0$ and obtain the temperature-entropy relation $T_i = e^{S_i}$, $i=1,2$.

\begin{figure}[t]
	\centering
	\includegraphics[width=\linewidth]{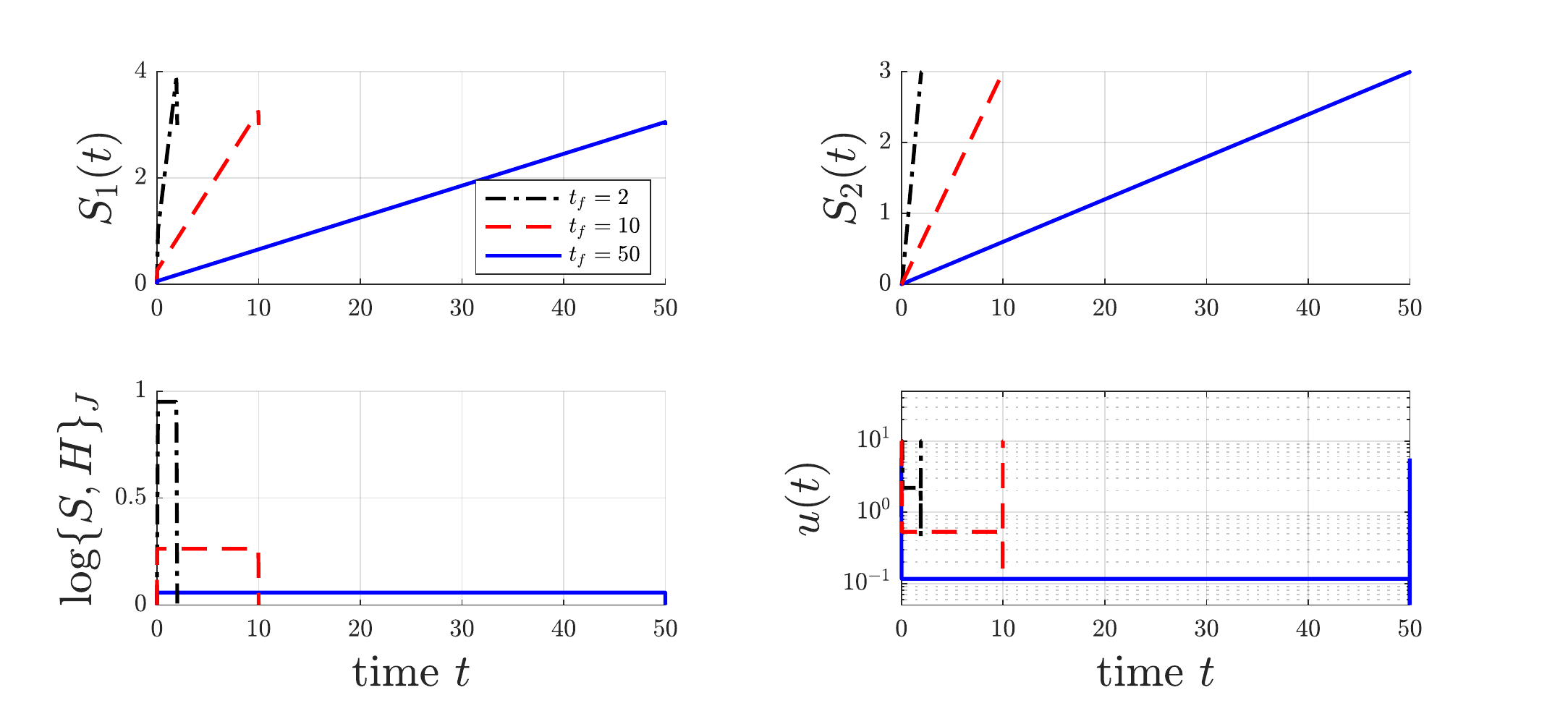}
	\includegraphics[width=0.8\linewidth]{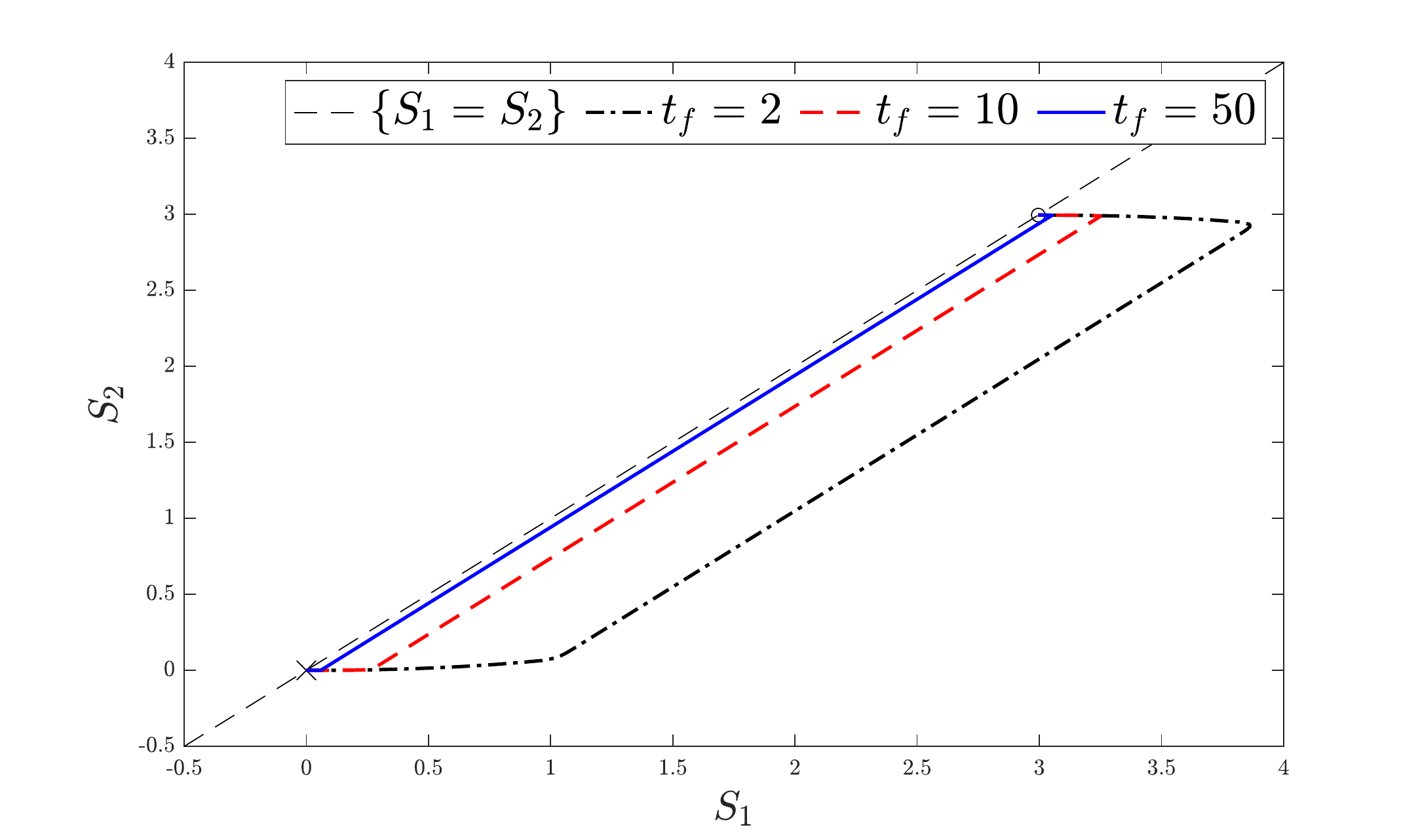}
	\caption{Depiction of the optimal intensive variable $S$ over time (top) and in phase space (bottom).}
	\label{fig:entropycontrol_overtime_ent}
\end{figure}

We consider the control constraint set  $\mathbb{U}  = [-10,10]$ and aim to perform a state transition within two thermodynamical equilibria:
$
T_1^0 = T_2^0 = 1 \quad \text{and} \quad T_1^{t_f}=T_2^{t_f} = 20$. 

In terms of the optimal control problem~\eqref{eq:phOCP}, this corresponds to the initial value in entropy variables $x^0 = \begin{pmatrix}
0&0
\end{pmatrix}^\top$ and the terminal set $\Psi = \left\{\begin{pmatrix}
\ln{20}&\ln{20}
\end{pmatrix}^\top\right\}$, by means of the relation $T_i = e^{S_i}$, $i=1,2$.

It is clear that the state transition is only possible through providing heat~--- or, equivalently, entropy~---~to the first compartment, cf.\ Figure~\ref{fig:heat}.

In Figure~\ref{fig:entropycontrol_overtime_ent}, we observe the distance of optimal state trajectories to the set of thermodynamic equilibria. We can not steer the system from initial to terminal state on this set, as by the form of the input vector in \eqref{eq:entcont}, no control action that is non-zero leaves $\mathcal{T}$ invariant. However, for increasing time horizons, the state trajectories remain closer and closer to the manifold, as the necessary control action, that is, the entropy flow, can be chosen smaller and smaller. Furthermore, we observe a turnpike behavior of the control towards zero as this is the only control that leaves the set of thermodynamic equilibria invariant, cf.\ Remark~\ref{rem:ss}.
The states, i.e., the individual entropies in the compartments depicted in the upper plot of Figure~\ref{fig:entropycontrol_overtime_ent} indicate a velocity turnpike, cf.\ \cite{FaulFlasOberWort21,pighin2020turnpike}, that is, their velocity is constant for the majority of the time interval. This can be explained as the Poisson bracket is mostly constant and small --- the state has a turnpike towards the manifold --- and the control is mostly constant and small --- the zero control is the only control that leaves this manifold invariant --- and thus the dynamics \eqref{eq:RIPHS} imply that $\dot{x}_1 = \dot{S}_1\approx \text{const.}$ and $\dot{x}_2=\dot{S}_2\approx \text{const.}$.

We depict the corresponding quantitites extensive variable $H_x(S_1,S_2)=(T_1,T_2)$ in Figure~\ref{fig:entropycontrol_overtime_temp}. Here, we observe---due to the algebraic relation $T=\exp(S)$---an exponential behaviour in the upper plot of Figure~\ref{fig:entropycontrol_overtime_temp}. In the lower plot, we can observe that, the temperature is moving further and further away from the subspace. The reason is that, by means of the turnpike property, we have an optimal rate of travel in the state variable, that is, e.g. for the first compartment, $\text{const.}=\dot{S}_1 = \tfrac{T_1-T_2}{T_1}$. For increasing temperatures $T_1$, the latter fraction can only be constant if also $T_1-T_2$ increases.
\begin{figure}[t]
	\centering
	\includegraphics[width=\linewidth]{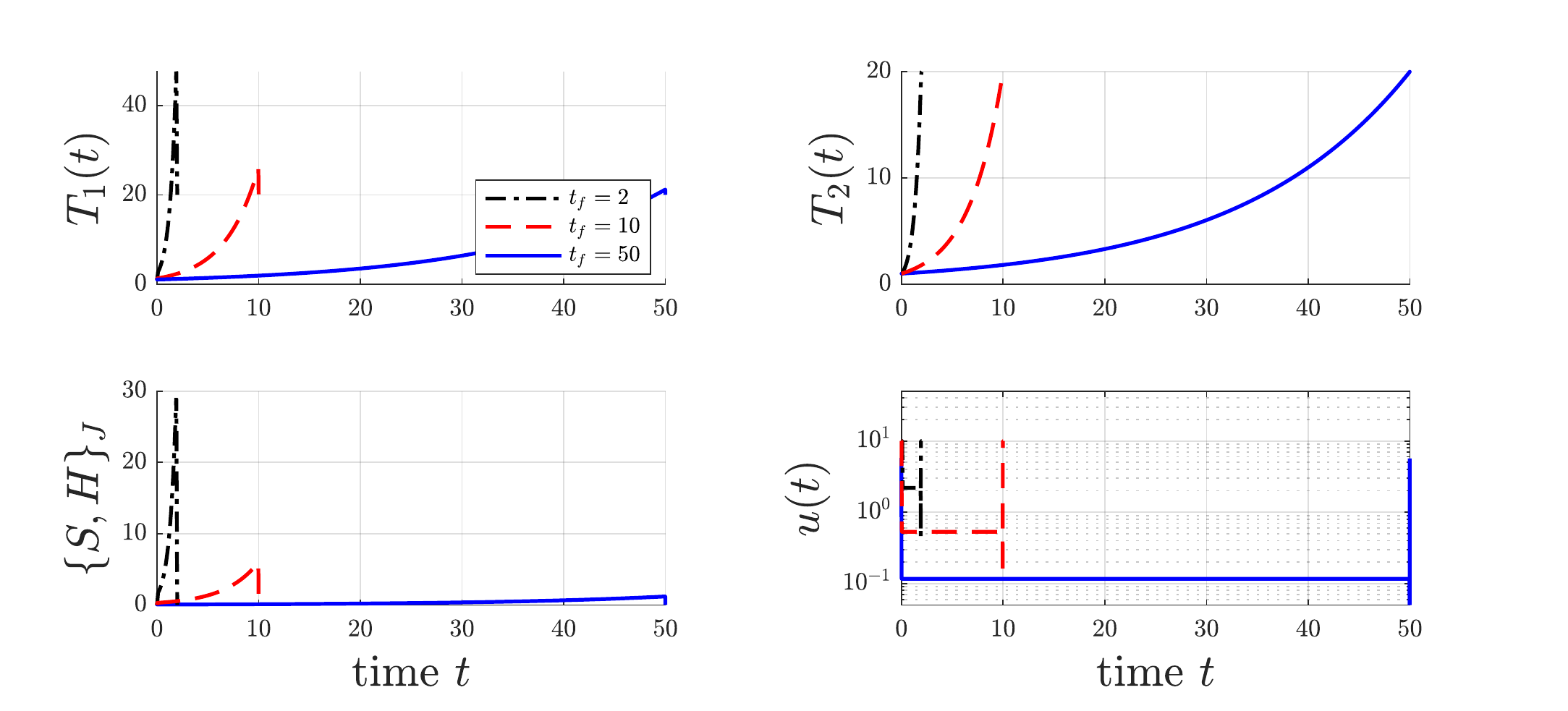}
	\includegraphics[width=.8\linewidth]{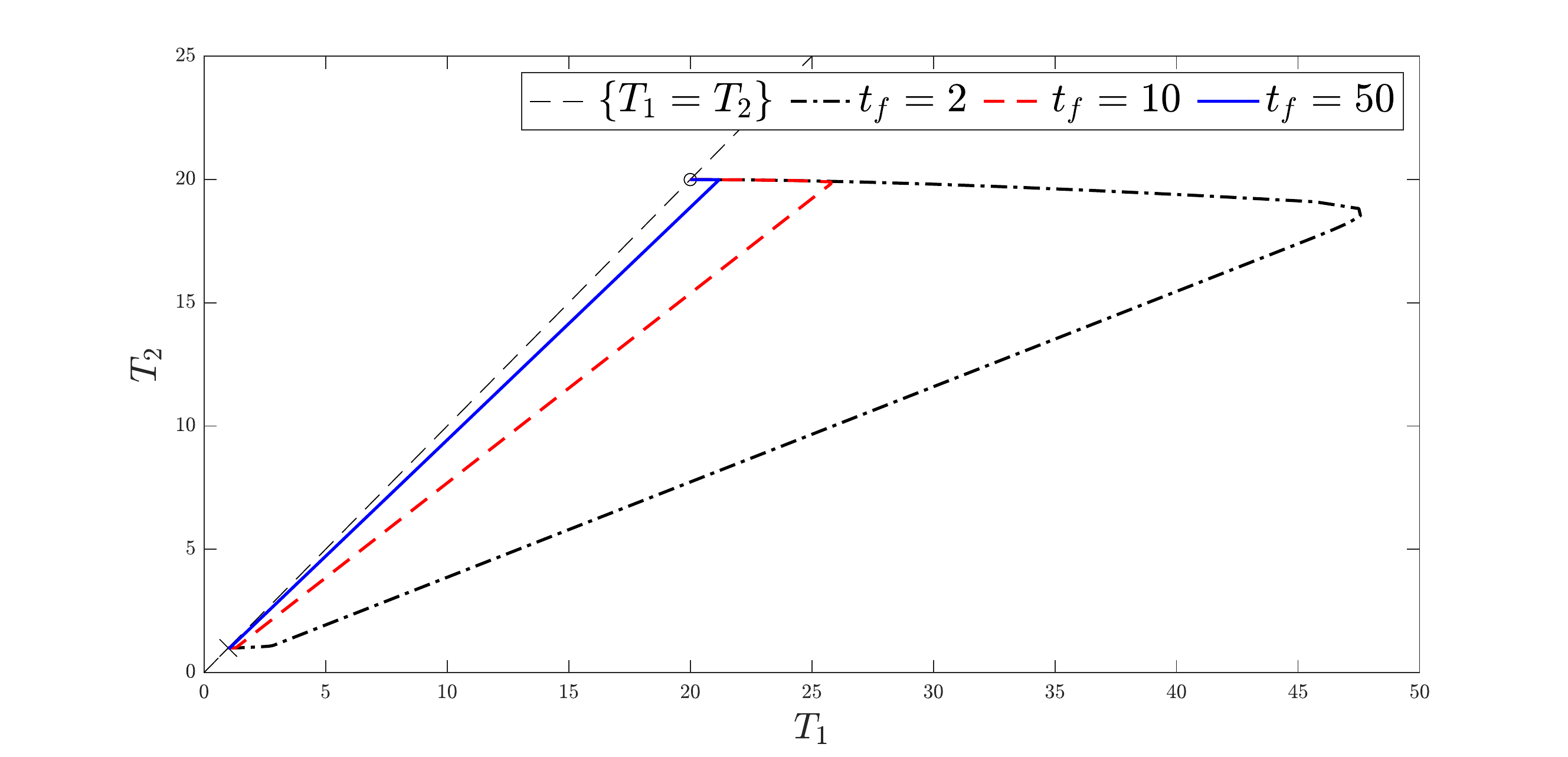}
	\caption{Depiction of the optimal extensive variable $T = H_x$ over time (top) and in phase space (bottom).}
	\label{fig:entropycontrol_overtime_temp}
\end{figure}

\section{Discussion and conclusion}\label{subsec:disc}

This paper considered optimal state transitions for irreversible port-Hamiltonian systems with minimal supply of energy, entropy, or exergy. We proved a manifold turnpike property w.r.t.\ 
the set of thermodynamical equilibria. Moreover, we numerically illustrated our findings drawing upon the example of a heat exchanger.

Future research will consider the relaxation of Assumption~\ref{as:comp} used for proving the turnpike theorem. To this end, we aim to show sufficient conditions via a refined argument using optimality, dissipativity, and properties of the input vector field. Put differently, we will investigate the lack of forward invariance of the manifold of thermodynamic equilibria which is related to $g(x,H_x)\notin T_x\mathcal{T}$ where $T_x\mathcal T$ is the tangent space of the manifold $\mathcal{T}$ at $x$. 



\bibliographystyle{abbrv}
\bibliography{references.bib}

\appendix
\section*{Appendix -- Proof of Lemma~\ref{lem:manifold_distance}}
\noindent We will first prove that
	\begin{align}\label{e:proofeq1}
	c\|x_1-x_2\|\leq \left\|\tfrac{\partial H}{\partial x}(x_1)-\tfrac{\partial H}{\partial x} (x_2)\right\|\leq C\|x_1-x_2\|
	\end{align}
	holds for all $x_1,x_2\in K$. To show the second inequality, we compute
	\begin{align*}
	\left\| H_x(x_1)-H_x (x_2)\right\|&\le\Big(\sup_{\xi\in\operatorname{conv}(K)}\left\|H_{xx}(\xi)\right\|\Big)\cdot\|x_1-x_2\| \\&= C\|x_1-x_2\|, 
	\end{align*}
	where $\operatorname{conv}(K)$ denotes the convex hull of $K$. For the first inequality of \eqref{e:proofeq1} we note that the inverse $[H_{xx}(x)]^{-1}$ is continuous and thus in particular bounded on the compact set $H_x^{-1}\left(\conv(H_x(K))\right)$. Hence, for $z_1,z_2\in H_x(K)$ we have
	\begin{align*}
	&\left\|H_x^{-1}(z_1)-H_x^{-1}(z_2)\right\|
	\\&\le\Big(\sup_{\xi\in\conv(\frac{\partial H}{\partial x}(K))}\left\|(H_x^{-1})'(\xi)\right\|\Big)\|z_1-z_2\|\\
	&= \Big(\sup_{\xi\in\conv(H_x(K))}\left\|\left[H_{xx}(H_x^{-1}(\xi))\right]^{-1}\right\|\Big)\|z_1-z_2\| 
	\\&=  c^{-1}\|z_1-z_2\|,
	\end{align*}
	which is equivalent to $\|H_x(x_1)-H_x(x_2)\|\ge c\|x_1-x_2\|$. This proves \eqref{e:proofeq1}.\\
	Now, for $x,v\in\R^n$ we have $|v^\top x| = \|v\|\dist(x,v^\perp)$ and therefore
	\begin{align*}
	\big|\{S(x),H(x)\}_{J}\big|
	&=|H_x(x)^\top Jl|
	= \|Jl\|\dist\left(H_x(x), (Jl)^\perp\right) \\&
	= \|Jl\|\inf_{z\in (Jl)^\perp}\left\| H_x(x)-z\right\|\\
	&= \|Jl\|\inf_{z\in (Jl)^\perp}\left\|H_x(x) - H_x(H_x^{-1}(z))\right\|.
	\end{align*}
	Hence, by \eqref{e:proofeq1},
	\begin{align*}
	\big|\{S(x),H(x)\}_{J}\big| &\le C\inf_{z\in (Jl)^\perp}\|x-H_x(x)^{-1}(z)\| \\&= C\inf_{w\in \mathcal{T}}\|x-w\| = C\dist(x,\mathcal{T})
	\end{align*}
	and
	\begin{align*}
	\big|\{S(x),H(x)\}_{J}\big| &\ge c\inf_{z\in (Jl)^\perp}\|x-H_x(x)^{-1}(z)\| \\&= c\inf_{w\in \mathcal{T}}\|x-w\| = c\dist(x,\mathcal{T}).
	\end{align*}
	The claim now follows from the fact that $\gamma$ is continuous and positive and $K$ is compact.\qed

\end{document}